\documentclass[12pt]{amsart}\usepackage{amsmath,amsthm,amssymb,cite}

\DeclareMathOperator{\RE}{Re}

\numberwithin{equation}{section}
\newtheorem{theorem}{Theorem}[section]
\newtheorem{lemma}[theorem]{Lemma}
\newtheorem{corollary}[theorem]{Corollary}
\theoremstyle{remark}
\newtheorem{remark}[theorem]{Remark}
\newtheorem{example}[theorem]{Example}

\makeatletter
\@namedef{subjclassname@2010}{%
  \textup{2010} Mathematics Subject Classification}
\makeatother

\begin{document}
\author[SUBZAR BEIG]{SUBZAR BEIG}

\address{Department of Mathematics, University of Delhi,
Delhi--110 007, India}
\email{beighsubzar@gmail.com}
\author[V. Ravichandran]{V. Ravichandran}

\address{Department of Mathematics, University of Delhi,
Delhi--110 007, India}
\email{vravi68@gmail.com}

\title[Convolution of a harmonic mapping]{Convolution  of a harmonic mapping with $n$-starlike mappings and its partial sums}

\begin{abstract} We investigate the univalency and  the directional convexity of the convolution $\phi\tilde{*}f=\phi*h+\overline{\phi*g}$ of the harmonic mapping $f=h+\bar{g}$ with a mapping $\phi$ whose convolution with the mapping $z+\sum_{k=2}^{\infty}k^nz^k$ is starlike (and such a mapping $\phi$ is called $n$-starlike). In addition, we investigate the directional convexity of (i)  the convolution of an analytic  convex mapping with the slanted  half-plane mapping, and (ii)  the partial sums of the convolution of a $6$-starlike mapping with the harmonic Koebe mapping  and the harmonic half-plane mapping.
\end{abstract}

\thanks{The first author is supported by  a Junior Research Fellowship from UGC, New Delhi, India.}
\maketitle

\section{Introduction}
Let $\mathcal{H}$ consists of all complex-valued harmonic mappings $f=h+\bar{g}$ in the unit disk $\mathbb{D}:=\left\lbrace z\in\mathbb{C}, |z|<1\right\rbrace$, where $h$ and $g$ are analytic mappings. Let  $\mathcal{S}_H^0$ be the sub-class of $\mathcal{H}$ consists of all mappings $f$  in the class $\mathcal{H}$ that are univalent, sense-preserving and normalized by the conditions  $f(0)=f_{\bar{z}}(0)=f_z(0)-1=0$. Let $\mathcal{K}_H^0$ and $\mathcal{S}_H^{*0}$ denote  the sub-classes of $\mathcal{S}_H^0$ consisting of mappings  which maps $\mathbb{D}$ onto convex and starlike domains respectively. The sub-classes $\mathcal{S}$, $\mathcal{K}$ and $\mathcal{S}^*$ of analytic mappings consisting of univalent, convex and starlike mappings are respectively sub-classes of $\mathcal{S}_H^0$, $\mathcal{K}_H^0$ and $\mathcal{S}_H^{*0}$. Clunie and Sheil-Small\cite{1} constructed two important mappings, the harmonic right-half plane mapping $L \in \mathcal{K}_H^{*0} $ and the harmonic Koebe mapping $K\in{\mathcal{S}_H^0}$. These mappings are expected to play the role   of extremal mappings respectively in classes $\mathcal{S}_H^0$ and $\mathcal{K}_H^0$ as played by the analytic right-half plane mapping and analytic Koebe mapping respectively in the classes $\mathcal{K}$ and $\mathcal{S}^*$. These mappings $K=H+\overline{G}$ and $L=M+\overline{N}$ are defined in $\mathbb{D}$ by
 \[H(z)=\frac{z-z^2/2+z^3/6}{(1-z)^3},\quad G(z)=\frac{z^2/2+z^3/6}{(1-z)^3}\] and
\[M(z)=\frac{z-z^2/2}{(1-z)^2}, \quad N(z)=\frac{-z^2/2}{(1-z)^2}.\]
 A domain  $D$ is said to be \textit{convex in direction ${\theta}$} $(0\leq\theta<2\pi),$ if every line parallel to the line joining $0$ and  $\emph{e}^{\textit{i}\theta}$ lies completely inside or outside the domain $D$. If $\theta=0$ ( or $\pi/2$), such a domain $D$ is called convex in the direction of real (or imaginary) axis. In this paper we study the directional convexity of the convolution of these and some other mappings with $n$-starlike mapping introduced by S\~{a}l\~agean\cite{2} and their partial sums. Let $\mathcal{A}$ be the class of all analytic mappings $f:\mathbb{D}\longrightarrow\mathbb{C}$ with $f(0)=0$, and $ f'(0) = 1$. The function $f\in\mathcal{A}$ has the Taylor series expansion $f(z)=z+\sum_{k=2}^{\infty}a_kz^k$.
For the function $f \in \mathcal{A}$, $n\geq0$,  S\~{a}l\~{a}gean \cite{2}  defined the differential operator $\mathcal{D}^n:\mathcal{A}  \longrightarrow \mathcal{A}$ by
\[\mathcal{D}^nf(z)=z+\sum_{k=2}^{\infty}k^na_kz^k.\]
By using this operator,  S\~{a}l\~{a}gean introduced the class of $n$-starlike mappings of order $\alpha$ $(0\leq{\alpha<{1}})$ defined by\[\mathcal{S}_n(\alpha):= \bigg\{f\in \mathcal{A}: \RE{\frac{\mathcal{D}^{n+1}f(z)}{\mathcal{D}^{n}f(z)}>{\alpha},\quad z\in \mathbb{D} }\bigg\}.\]
Equivalently, a function $f\in\mathcal{A}$ is $n$-starlike of order $\alpha$ if and only if the function $\mathcal{D}^nf$ is starlike of order $\alpha$. Clearly  $\mathcal{S}^*(\alpha)= \mathcal{S}_0(\alpha)$ and $\mathcal{K}(\alpha)= \mathcal{S}_1(\alpha)$ are respectively the classes of starlike and convex mappings of order $\alpha$ introduced by Robertson\cite{21}. Denote $\mathcal{S}_n(0)$ by $\mathcal{S}_n$ and the mappings in this class are called $n$-starlike mappings. Also, $\mathcal{S}^*= \mathcal{S}_0$ and $\mathcal{K}= \mathcal{S}_1$.
Suppose $\phi$ and $\psi$ are analytic mappings on  $\mathbb{D} $ with  $\phi(z)=\sum_{n=1}^\infty a_nz^n$ and $ \psi(z)=\sum_{n=1}^\infty b_nz^n,$
their convolution $\phi *\psi$ is defined by $(\phi *\psi)(z):=\sum_{n=1}^\infty a_n b_nz^n.$ Convolution of harmonic mappings $f=h+\bar{g}$ and $F=H+\overline{G}$ is defined by $f*F:= h*H+ \overline{g*G}.$ Also,   the convolution $\tilde{*}$ of harmonic mapping $f=h+\bar{g}$ with analytic mapping $\phi$ is defined by$ f\tilde{*}\phi:=h*\phi+\overline{g*\phi}.$
It is well known that the convolution of two harmonic convex mappings  is not necessarily convex/univalent. In \cite{6}, Dorff studied the directional convexity of harmonic mappings and proved that convolution of two right half-plane mappings is univalent and convex in the direction of real axis provided the convolution is locally univalent. Later, Dorff \textit{et al.} \cite{7} extended such results to slanted half-plane and strip mappings. Other recent related work in this direction can be found in \cite{1,2,5,9,10,11,13,14,15,16,17,18,19,20,25}. For analytic convex mappings $\phi$, the convolution $K\tilde{*}\phi$ is not necessarily univalent. However, Nagpal and Ravichandran \cite{8} showed that $K\tilde{*}\phi\in\mathcal{S}_H^0$ and is convex in the direction of real axis, if $\phi$ is a $2$-starlike mapping.\

In Section 2, we prove that the convolution of certain harmonic mappings with $n$-starlike mappings is univalent and  convex in a particular direction. In particular, for $0\leq\alpha<\pi$, we prove that the convolution of an analytic  convex mapping with the slanted half-plane mapping is univalent and convex in the direction of $\pi/2-\alpha$. Lastly, in Section 3, we discuss the partial sums of $n$-starlike mappings and prove that all the partial sums of  $n$-starlike mappings with $n\geq 4$ are $(n-4)$-starlike. By using this, we prove that all the partial sums of the convolution of $6$-starlike mappings with the mappings $L$ and $K$ are univalent and convex in the direction of real-axis.

\section{Convolution of some harmonic mappings with $n$-starlike mappings }
We first give some convolution properties of $n$-starlike mappings, which will be useful throught the paper. From the  defnition of $\mathcal{S}_n(\alpha)$, one can easily see that\begin{equation}\label{eq0}
 f\in\mathcal{S}_n(\alpha)\Leftrightarrow\mathcal{D}^{n-m}f\in\mathcal{S}_m(\alpha).
  \end{equation}
Using this relation, we get the following result regarding the convolution of mappings in class $\mathcal{S}_n$.

 \begin{lemma}\label{theom1a}
Let $n + m\geq{1}$. If the function $f \in \mathcal{S}_n$ and the function $g\in \mathcal{S}_m$, then the convolution $f*g\in \mathcal{S}_{n+m-1}.$
 \end{lemma}

 \begin{proof}
Assume that $n\geq1$.  Since the function $f \in \mathcal{S}_n$  and the function $g\in \mathcal{S}_m$,  by \eqref{eq0}, the function $\mathcal{D}^{n-1}f\in \mathcal{K}$  and  the function $\mathcal{D}^{m}g\in \mathcal{S}^*$.  Therefore, from \cite{22}, we have\[\mathcal{D}^{n+m-1}(f*g)=\mathcal{D}^{n-1}f*\mathcal{D}^{m}g\in\mathcal{S}^*.\] Hence, by \eqref{eq0}, it follows that the convolution $f*g\in \mathcal{S}_{n+m-1}$.
 \end{proof}

 \begin{theorem}\label{theomd}\cite{24}
If the function $f(z)=z+\sum_{n=2}^\infty a_n z^n\in \mathcal{A}$  satisfies the inequality  $\sum_{n=2}^\infty (n-\alpha)|a_n|\leq1-\alpha,$ then the function $f$ is starlike of order $\alpha.$
\end{theorem}

By using \eqref{eq0} and Theorem \ref{theomd}, we get the following result.

\begin{theorem}\label{theomd2}\cite{24}
Let $0\leq\alpha<1$. If the function $f(z)=z+\sum_{n=2}^\infty a_n z^n$ in $\mathbb{D}$ satisfies the inequality  $\sum_{n=2}^\infty n^{m-1}(n-\alpha)|a_n|\leq1-\alpha,$ then the function $f$ is $(m-1)$-starlike of order $\alpha.$
\end{theorem}

 The harmonic mappings $f$ considered in this paper are assumed to be normalized by$f(0)=f_{z}(0)-1=f_{\bar{z}}(0)=0$, unless otherwise specified. The following result due to Clunie and Sheil-Small is used for constructing univalent harmonic mappings convex in a given direction.

\begin{lemma}\label{lem130}\cite{3}
A locally univalent harmonic mapping $f=h+\overline{g}$ on $\mathbb{D}$ is univalent and maps $\mathbb{D}$ onto a domain convex in the direction of $\phi$ if and only if the analytic mapping $h-\emph{e}^{2\textit{i}\phi}g$ is univalent and maps $\mathbb{D}$ onto a  domain convex in the direction of $\phi$.
\end{lemma}

\begin{lemma} \label{lem 13a}
Let the function $ f= h+ \bar {g} $ be harmonic and  the function $\phi$ be analytic in $\mathbb{D}$. If  the function $ (h-e^{-2\textit{i}\beta}g)*\phi$ is convex  and,  for some real number $\gamma$,
\begin{equation}\label{eq11}
\RE\frac{(\phi*h)'(z)}{\left((\phi*h)'-e^{-2\textit{i}\gamma}(\phi*g)'\right)(z)}>\frac{1}{2}  \quad \text {for z} \in \mathbb{D},
\end{equation}
then the convolution $f\tilde{\ast}\phi \in \mathcal{S}_H^0$ and is convex in the  direction of $-\beta$.
\end{lemma}

\begin{proof}
Since the function $(h-e^{-2\textit{i}\beta}g)*\phi$ is convex and hence convex in the direction of $-\beta$, in view of Lemma \ref{lem130}, it is enough to prove that the mapping $f\tilde{\ast}\phi$ is locally univalent. Clearly \eqref{eq11} shows that $(h*\phi)'(z)\neq0$ for $z\in \mathbb{D}$. Therefore, using \eqref{eq11}, we see that the dilatation $w_{e^{\textit{i}\gamma}f\tilde{\ast}\phi}=(e^{-\textit{i}\gamma}g*\phi)'/(e^{\textit{i}\gamma}h*\phi)'$ of $e^{\textit{i}\gamma}f\tilde{\ast}\phi$ satisfies \begin{align}
\RE\frac{1+w_{e^{\textit{i}\gamma}f\tilde{\ast}\phi}(z)}{1-w_{e^{\textit{i}\gamma}f\tilde{\ast}\phi}(z)}&= \RE\frac{(e^{\textit{i}\gamma}h*\phi)'(z)+(e^{-\textit{i}\gamma}g*\phi)'(z)}{(e^{\textit{i}\gamma}h*\phi)'(z)-(e^{-\textit{i}\gamma}g*\phi)'(z)}\notag\\&=2\RE\frac{(\phi*h)'(z)}{\left((\phi*h)'-e^{-2\textit{i}\gamma}(\phi*g)'\right)(z)}-1>\frac{1}{2}, \quad z\in \mathbb{D}.\label{eq12}
\end{align}This shows that $|w_{e^{\textit{i}\gamma}f\tilde{\ast}\phi}(z)|<1$ for $z\in\mathbb{D}$, or equivalently $|w_{f\tilde{\ast}\phi}(z)|<1$ for $z\in\mathbb{D}$, where $w_{f\tilde{\ast}\phi}=(g*\phi)'/(h*\phi)'$ is dilatation of the function $f\tilde{\ast}\phi$. The result now follows by Lewis theorem.
\end{proof}

\begin{theorem}\label{theom15a}
Let the  harmonic  function $ f= h+ \bar {g} $  in $\mathbb{D}$ satisfies $h(z)- e^{-2\textit{i}\gamma}g(z)=z$ for some real number $\gamma$. If the function $h*\phi\in\mathcal{S}_2$  and the function $(h-e^{-2\textit{i}\beta}g)*\phi\in\mathcal{K}$ for some analytic function $\phi$, then the convolution $f\tilde{*}\phi\in\mathcal{S}_H^0$ and is convex in the direction  of $-\beta$.
\end{theorem}

\begin{proof}
 Since the function $h*\phi\in\mathcal{S}_2$, we have $z(h*\phi)'\in\mathcal{K}$. Hence,  from \cite[Corollary 1, p.251]{23}, we get \[\RE\frac{(\phi*h)'(z)}{\left((\phi*h)'-e^{-2\textit{i}\gamma}(\phi*g)'\right)(z)}=\RE(h*\phi)'(z)>\frac{1}{2} \quad\text{for }z\in\mathbb{D}.\] Also, the function $(h-e^{-2\textit{i}\beta}g)*\phi$ is convex. The result now follows from Lemma \ref{lem 13a}.
\end{proof}

\begin{corollary}\label{corl15b}
If the function $\phi(z)=z+\sum_{n=2}^\infty a_nz^n\in\mathcal{S}_2$ satisfies the inequality $\sum_{n=2}^\infty n^2|a_n|\leq 1/\sqrt{2(1-\cos2\theta)}$ for some $\theta \in [0,\pi/2]$, then, for real number $\gamma$, the function $\phi+\overline{e^{2\textit{i}\gamma}(\phi-z)}\in\mathcal{S}_H^0$ and is convex in the direction  $-\beta$ for all  $\beta$  satisfying $|\beta+\gamma|\leq\theta$.
\end{corollary}
\begin{proof}
The harmonic mapping $f=h+\bar g$, with $h(z)=z/(1-z)$ and $g(z)=e^{2\textit{i}\gamma}(z/(1-z)-z)$, satisfy $h(z)-e^{-2\textit{i}\gamma}g(z)=z$. Also, the function $h*\phi=\phi\in\mathcal{S}_2.$  Furthermore, we see that
\begin{align*}
 (h(z)- e^{-2\textit{i}\beta}g(z))*\phi(z)&=\phi(z)-e^{-2\textit{i}(\beta+\gamma)}(\phi(z)-z)\\&=z+\sum_{n=2}^\infty(1-e^{-2\textit{i}(\beta+\gamma)})a_nz^n
\end{align*}
Now, for $|\beta+\gamma|\leq\theta$, we have
\begin{align*}
\sum_{n=2}^\infty n^2|(1-e^{-2\textit{i}(\beta+\gamma)})a_n|&=|1-e^{-2\textit{i}(\beta+\gamma)}|\sum_{n=2}^\infty n^2|a_n|\\&=\sqrt{2(1-\cos2(\beta+\gamma))}\sum_{n=2}^\infty n^2|a_n|\\&\leq\sqrt{2(1-\cos2(\beta+\gamma))}\frac{1}{\sqrt{2(1-\cos2\theta)}}\leq1.
\end{align*}
 Therefore, by Lemma \ref{theomd2}, the function $(h- e^{-2\textit{i}\beta}g)*\phi$ is convex for every $\beta$ such that $|\beta+\gamma|\leq\theta$. Hence, by Theorem \ref{theom15a}, the function $f\tilde{*}\phi=\phi+\overline{e^{2\textit{i}\gamma}(\phi-z)}\in\mathcal{S}_H^0$ and is convex in every direction $-\beta$ satisfying $|\beta+\gamma|\leq\theta$.
\end{proof}
\remark\label{remaka1} For $\theta=\pi/2$ in Corollary \ref{corl15b}, we see the function $\phi+\overline{e^{2\textit{i}\gamma}(\phi-z)}\in\mathcal{K}_H^0$, if the function  $\phi=z+\sum_{n=2}^\infty a_nz^n\in\mathcal{S}_2$ and satisfies the inequality $\sum_{n=2}^\infty n^2|a_n|\leq 1/2.$
\remark\label{remaka2}If the function  $\phi=z+\sum_{n=2}^\infty a_nz^n$ satisfies the inequality $\sum_{n=2}^\infty n^3|a_n|\leq 1$, then $\sum_{n=2}^\infty n^2|a_n|\leq 1/2$ and, by Lemma \ref{theomd2}, the function $\phi \in \mathcal{S}_2$. Therefore, Remark \ref{remaka1} shows that $\phi+\overline{e^{2\textit{i}\gamma}(\phi-z)}\in\mathcal{K}_H^0$.
\begin{theorem}\label{theom16a}Let the function $\phi\in\mathcal{S}_2$ and the function $f= h+ \bar {g} $ be  a harmonic mapping in $\mathbb{D}$ satisfying $h(z)- e^{-2\textit{i}\gamma}g(z)=h(z)*\log(1/(1-z))$ for some real number $\gamma$ with the function $h\in\mathcal{S}^*$. Then the convolution $f\tilde{*}\phi\in\mathcal{S}_H^0$ and is convex in the direction $-\gamma$. Furthermore, if for any $\beta$ real, the function $h-e^{-2\textit{i}\beta}g\in\mathcal{S}^*$, then the convolution $f\tilde{*}\phi$ is convex in the direction $-\beta$.
\end{theorem}
\begin{proof}
Since the function $h\in\mathcal{S}^*$ and the function $\log(1/(1-z))\in\mathcal{S}_2$, by Lemma \ref{theom1a}, we have the function $h-e^{-2\textit{i}\gamma}g=h*\log1/(1-z)\in\mathcal{K}$. Hence, from \cite[Corollary 1, p.251]{23}, we have
 \begin{equation}\label{eq16b}
\RE\frac{h(z)}{h(z)*\log\frac{1}{1-z}}>\frac{1}{2}, \quad z \in\mathbb{D}.
\end{equation}
Note that, we can write
\begin{equation}\label{eq16c}
\RE\frac{(\phi*h)'(z)}{\left((\phi*h)'-e^{-2\textit{i}\gamma}(\phi*g)'\right)(z)} =\RE\frac{z\phi'(z)*(h(z)*\log\frac{1}{1-z})\left(\frac{h(z)}{h(z)*\log(\frac{1}{1-z})}\right)}{z\phi'(z)*(h(z)*\log\frac{1}{1-z})},
\end{equation}
 where the function $z\phi'\in\mathcal{K}$ and the function $h*\log1/(1-z)\in\mathcal{S}^*$. Therefore, in view of $\eqref{eq16b}$,  $\eqref{eq16c}$ and \cite[Theorem 2.4, p.54]{8}, it  follows that \begin{equation}\label{eq16d}
\RE\frac{(\phi*h)'(z)}{\left((\phi*h)'-e^{-2\textit{i}\gamma}(\phi*g)'\right)(z)}>\frac{1}{2}, \quad z\in \mathbb{D}.
\end{equation}
As the function $h\in\mathcal{S}^*$ and the function $\phi\in\mathcal{S}_2$, Lemma \ref{theom1a} gives that the function $(h-e^{-2\textit{i}\gamma}g)*\phi=h*\log(1/(1-z))*\phi\in\mathcal{K}$. Similarly, the function $(h-e^{-2\textit{i}\beta}g)*\phi\in\mathcal{K}$. Therefore, in view of \eqref{eq16d}, the result follows by Lemma \ref{lem 13a}.
\end{proof}
\begin{corollary}\label{corla1}Let the function $\phi\in\mathcal{S}_2$ and the function $f= h+ \bar {g} $ be  a harmonic mapping in $\mathbb{D}$ satisfying $h(z)- g(z)=h(z)*\log(1/(1-z))$. Then, we have the following.\begin{itemize}
\item[1] If the function $h\in\mathcal{S}^*$, then the convolution $f\tilde{*}\phi\in \mathcal{S}_H^0$ and is convex  in the direction of real axis.
\item[2] If, for some $\theta$ $(0\leq\theta<\pi)$, the function $h(z)=z+\sum_{n=2}^\infty a_nz^n\in\mathcal{S}^*$  satisfies the inequality $\sum_{n=2}^\infty |a_n|\sqrt{2n(n-1)(1-\cos2\theta)+1}\leq1$, then the convolution $f\tilde{*}\phi\in \mathcal{S}_H^0$ and is convex  in every  direction $-\beta$ such that $|\beta|\leq \theta$.
\item[3] If, for some $\theta$ $(0\leq\theta<\pi)$ such that $\cos2\theta\leq1/4$, the function $h(z)=z+\sum_{n=2}^\infty a_nz^n$ satisfies the inequality $\sum_{n=2}^\infty |a_n|\sqrt{2n(n-1)(1-\cos2\theta)+1}\leq1$, then the convolution $f\tilde{*}\phi\in \mathcal{S}_H^0$ and is convex  in every  direction $-\beta$ such $|\beta|\leq \theta$
\end{itemize}
\end{corollary}
\begin{proof}
$(1)$ is obvious from Theorem \ref{theom16a}. For $(2)$  and $(3)$, in view of Theorem \ref{theom16a}, it is enough to prove that the function $h-e^{-2\textit{i}\beta}g\in\mathcal{S}^*$ for every $\beta$ such that $|\beta|\leq \theta$ and  $h\in\mathcal{S}^*$. If $\cos2\theta\leq1/4$, then the inequality $\sum_{n=2}^\infty |a_n|\sqrt{2n(n-1)(1-\cos2\theta)+1}\leq1$ implies the inequality $\sum_{n=2}^\infty n|a_n|\leq1$. Hence, by Theorem \ref{theomd}, the function $h\in\mathcal{S}^*$.  Also, for $\beta$ real, we have \begin{align*}
h(z)-e^{-2\textit{i}\beta}g(z)&=h(z)-e^{-2\textit{i}\beta}(h(z)-h(z)*\log(1/1-z))\\&=z+\sum_{n=2}^\infty a_n(1-e^{-2\textit{i}\beta}(n-1)/n ).
\end{align*}
Therefore, in  view of $(2)$ and $(3)$, we see, for $|\beta|\leq\theta$, that the function $h-e^{-2\textit{i}\beta}g$ satisfies the inequality  \begin{align*}
\sum_{n=2}^\infty n| a_n(1-e^{-2\textit{i}\beta}(n-1/n)|&=\sum_{n=2}^\infty | a_n|\sqrt{2n(n-1)(1-\cos2\beta)+1}\\&\leq\sum_{n=2}^\infty | a_n|\sqrt{2n(n-1)(1-\cos2\theta)+1}\leq1.
\end{align*}Hence, by Theorem \ref{theomd}, the function $h-e^{-2\textit{i}\beta}g\in\mathcal{S}^*$.
\end{proof}

\remark\label{remaks1} Taking $\theta=\pi/2$ in  Corollary \ref{corla1}, we see that, if the function $\phi\in\mathcal{S}_2$ and the function $h(z)=z+\sum_2^\infty a_nz^n$ satisfies the inequality $\sum_2^\infty (2n-1)|a_n|\leq1$, then the convolution  $f\tilde{*}\phi\in\mathcal{K}_H^0$, where the function $f=h+\bar g$ and the function $g(z)=h(z)-h(z)*\log(1/(1-z)) $.

\begin{example}Let the harmonic function $f=h+\bar {g}$ be  given by the function $h(z)=z/(1-z)^2$ and the function $g(z)=z/(1-z)^2-z/(1-z)=z/(1-z)^2-z/(1-z)^2*\log(1/(1-z))$.  Now, taking the function $\phi(z)=\log(1/(1-z))$ in Theorem \ref{theom16a}, we get the convolution $f(z)*\log(1/(1-z))=z/(1-z)+\overline{z/(1-z)-\log(1/(1-z))}\in\mathcal{S}_H^0$ and is convex in the direction of real axis.
\end{example}

\begin{example}Let the harmonic function $f=h+\bar {g}$ be given by the function $h(z)=z+z^2/3$ and the function $g(z)=z^2/6$ and let the function $\phi(z)=\log(1/(1-z))$. Then, the functions $f$ and $\phi$ satisfy the conditions in Remark \ref {remaks1}, and hence the convolution $(f*\phi)(z)=z+z^2/6+\overline{z^2/12}\in\mathcal{K}_H^0$.
\end{example}

We denote the convolution $f*f*\dots*f$ ($n$-times) by $\left(f\right)_*^n$.
A simple calculation shows that, for $|\alpha|=1$ and $|z|<1$,  \[\left(\frac{z}{(1-z\alpha)^2}\right)^n_**\left(\alpha \log\frac{1}{1-z/\alpha}\right)_*^n=\frac{z}{1-z}\quad\text{ for all }  n\in \mathbb{N}.\]
Since the function $z/(1-z)$ is convolution identity, the inverse under convolution of  the function $\left(\frac{z}{(1-z\alpha)^2}\right)^n_*$ is the function $\left(\alpha \log\frac{1}{1-z/\alpha}\right)_*^{n}.$ For the  function $f$, we write the inverse of $\left(f\right)_*^n$ by $\left(f\right)_*^{-n}.$

\begin{theorem}\label{theom4a}
Let $n\in\mathbb{N}$, $|\alpha| =|\gamma|=1$ and $0\leq\beta<2\pi$. Let the  function $ f= h+ \bar {g} $ be a  harmonic mapping in $\mathbb{D}$ satisfying \[ h(z)-e^{-2\textit{i}\beta}g(z)=\frac{z}{(1-\alpha z)^2}*\bigg( \frac{z}{(1-\gamma z)^2}\bigg)_*^{n-2},\quad z \in\mathbb{D},\] and let the  function $\phi \in \mathcal{S}_n$. If the  function $ (h-e^{-2\textit{i}\delta}g)*\phi\in\mathcal{K}$ for some real number $\delta$ and
\begin{equation}\label{eq13}
\RE\left\lbrace\frac{(1-\alpha z)^2}{z}\left( h(z)*\bigg(\gamma \log\frac{1}{1- z/\gamma}\bigg)_*^{n-2}\right) \right\rbrace>\frac{1}{2},\quad z \in \mathbb{D},
\end{equation}
 then the convolution $f\tilde{\ast}\phi \in \mathcal{S}_H^0$ and is convex in the direction of $-\delta$.
\end{theorem}

\begin{proof}We have\[ (h(z)-e^{-2\textit{i}\beta}g(z))*\phi=\frac{z}{(1-\alpha z)^2}*\bigg( \frac{z}{(1-\gamma z)^2}\bigg)_*^{n-2}*\phi,\]  where the functions  $z/(1-\alpha z)^2$, $ z/(1-\gamma z)^2\in\mathcal{S}^*$ and  the function $\phi\in\mathcal{S}_n$. Therefore, by repeated application of Theorem \ref{theom1a}, we see that the function $(h-e^{-2\textit{i}\beta}g)*\phi\in\mathcal{K}.$
  Also, we have
 \begin{align}
 \lefteqn{\frac{(\phi*h)'(z)}{\left((\phi*h)'-e^{-2\textit{i}\beta}(\phi*g)'\right)(z)}}\notag\\
 & =\frac{z\phi'(z)*\left\lbrace h(z)*\bigg(\frac{z}{(1-\gamma z)^2}\bigg)^{n-2}_**\bigg(\gamma \log\frac{1}{1-z/\gamma}\bigg)_*^{n-2}\right\rbrace}{z\phi'(z)*(h-e^{-2\textit{i}\beta}g)(z)}\notag\\& =\frac{z\phi'(z)*\bigg(\frac{z}{(1-\gamma z)^2}\bigg)^{n-2}_**\frac{z}{(1-\alpha  z)^2} \left\lbrace \frac{(1-\alpha z)^2}{z}\left( h(z)*\bigg(\gamma \log\frac{1}{1- z/\gamma}\bigg)_*^{n-2}\right)\right\rbrace}{z\phi'(z)*\bigg(\frac{z}{(1-\gamma z)^2}\bigg)^{n-2}_**\frac{z}{(1-\alpha  z)^2}}.\label{eq13a}
\end{align}
Since the function $ \phi \in \mathcal{S}_n$ and the function $z/(1-\gamma z)^2\in\mathcal{S}^*$,  from Lemma \ref{theom1a}, we have \begin{equation}\label{eq13b}
 z\phi'(z) *\left(\frac{z}{(1-\gamma z)^2}\right)^{n-2}_* \in \mathcal{K}.
 \end{equation} Since the function $z/(1-\alpha z)^2 \in\mathcal{S}^*$,  in view of \eqref{eq13}, \eqref{eq13a} and \eqref{eq13b}, it follows, by \cite[Theorem 2.4, p.54]{8}, that \[\RE\frac{(\phi*h)'(z)}{\left((\phi*h)'-e^{-2\textit{i}\beta}(\phi*g)'\right)(z)}>1/2,\quad z\in\mathbb{D}.\] The result now follows from Lemma \ref{lem 13a}.
\end{proof}

\remark\label{remak38} Take $n=2$ in Theorem \ref{theom4a}.  Let $|\alpha| =1$, $0\leq\beta<2\pi$. Let the harmonic mapping $f=h+\bar g$  satisfy
 \begin{equation}\RE\label{eq1}
{\frac{(1-\alpha z)^2}{z}}h(z)> \frac{1}{2}, \quad  h(z)-e^{-2\textit{i}\beta}g(z)=\frac{z}{(1-\alpha z)^2 },  \quad z \in \mathbb{D},
\end{equation}
 and the function $ (h-e^{-2\textit{i}\delta}g)\in\mathcal{S}^*$ for some real number $\delta$. If the function $\phi \in \mathcal{S}_2$, then the convolution $f\tilde{\ast}\phi \in \mathcal{S}_H^0$ and is convex in the direction  $-\delta$, and in particular in the direction  $-\beta$.

\remark\label{remak39} Take $n=\gamma=1$ in Theorem \ref{theom4a} and notice that
\begin{align*}
\frac{z}{(1-\alpha z)^2}*\bigg( \frac{z}{(1-z)^2}\bigg)_*^{-1}=\frac{z}{(1-\alpha z)^2}* \log \frac{1}{1-z}=\frac{z}{(1-\alpha z)},
\end{align*}
and
\begin{align*}
 \frac{(1-\alpha z)^2}{z}\left( h(z)*\bigg(\log\frac{1}{1-z}\bigg)_*^{-1}\right)  = \frac{(1-\alpha z)^2}{z}\left( h(z)*\frac{z}{(1-z)^2}\right)= \frac{(1-\alpha z)^2}{z} \mathcal{D}h(z).
\end{align*}
Let $|\alpha| =1$, $0\leq\beta<2\pi$. Let the  harmonic mapping $f=h+\bar g$ satisfy
\begin{equation}
  h(z)-e^{-2\textit{i}\beta}g(z)=\frac{z}{(1-\alpha z)},\quad\RE\left\lbrace\frac{(1-\alpha z)^2}{z} \mathcal{D}h(z)\right\rbrace>\frac{1}{2},\quad z\in\mathbb{D},
\end{equation}
and the function $ h-e^{-2\textit{i}\delta}g\in\mathcal{K}$ for some real number $\delta$. If the function $\phi\in\mathcal{K}$, then the convolution $f\tilde{\ast}\phi \in \mathcal{S}_H^0$ and is convex in the direction $-\delta$, and in particular in the direction  $-\beta$.

\remark\label{remak40} Take $n=3$, $\gamma=1$ in Theorem \ref{theom4a}. First notice that
\[  \frac{z}{(1-\alpha z)^2}*\frac{z}{(1-z)^2} = \frac{z+\alpha z^2}{(1-\alpha z)^3}.\]
Let   $|\alpha| =1$, $0\leq\beta<2\pi$. Let the harmonic mapping $f=h+\bar g$ satisfy\[ h(z)-e^{-\textit{i}\beta}g(z)=  \frac{z+\alpha z^2}{(1-\alpha z)^3},\quad \RE\frac{(1-\alpha z)^2}{z}\left\lbrace h(z)*\log\frac{1}{1-z}\right\rbrace >1/2,\quad z\in\mathbb{D}.\]If the function $\phi\in\mathcal{S}_3$, then the convolution $f\tilde{\ast}\phi \in \mathcal{S}_H^0$ and is convex in the the direction $-\beta$.

Taking $\alpha=1$, $\beta=0$ in Remark \ref{remak38}, we get the following result.
\begin{corollary} \label{theom2a}\cite{8}
Let the function $\phi \in \mathcal{S}_2$ and the function $ f= h+ \bar {g} $ be a harmonic mapping in $\mathbb{D}$ satisfying $ h(z)-g(z)=z/(1-z)^2$ for all $z \in\mathbb{D}$. If\[\RE{\frac{(1-z)^2}{z}}h(z)> 1/2,   \quad \text{ for z} \in \mathbb{D},\]
then the convolution $f\tilde{\ast}\phi \in \mathcal{S}_H^0$ and is convex in the the direction of real axis.
\end{corollary}

Next, we give two examples of non-univalent convolution products.

\begin{example}\label{exam42}
 For $a\geq{-1}$ $(a\neq 0)$, consider the harmonic mapping $f_a=h+\bar{g}$ given by $h(z)=(1+ z/a)l(z)$ and $g(z)= zl(z)/a$, where $l(z)=z/(1-z)$ is analytic right-half plane mapping. Then, for the function $\phi(z)=z+z^2/2\in\mathcal{S}^*$,  we have
 \[(f_a\tilde{*}\phi)(z)=z+\frac{1+a}{2a}z^2+\frac{1}{2a}\bar{z}^2.\]Its Jacobian $\mathit{J}_{f_a\tilde{*}\phi}$,  given by\[\mathit{J}_{f_a\tilde{*}\phi}(z)=|(h*\phi)'(z)|^2-|(g*\phi)'(z)|^2=1+\frac{2+a}{a}|z|^2+2\frac{1+a}{a}\RE z,\]vanishes at $z=-a/(a+2)\in\mathbb{D}.$
\end{example}

\begin{example}\label{exam42a}
For $0<b<1/2$, consider the harmonic mapping $F_b=h+\bar{g}$ given by \[h(z)=\frac{z+(1+2b)z^2}{(1-z)^3}\quad\text{and}\quad g(z)= \frac{2bz^2}{(1-z)^3}.\]Then, for the function $\phi(z)=z+z^2/8\in\mathcal{S}_2$, the  Jacobian $\mathit{J}_{F_b\tilde{*}\phi}$ of the convolution $F_b\tilde{*}\phi$, given by\[\mathit{J}_{F_b\tilde{*}\phi}(z)=1+(1+b)|z|^2+(2+b)\RE z,\]vanishes at $z=-1/(1+b)\in\mathbb{D}.$
\end{example}

Examples \ref{exam42}-\ref{exam42a} show that if the function $\phi\in\mathcal{S}^*$ and $\mathcal{S}_2$, then respectively the convolutions  $f_a\tilde{*}\phi$ and  $F_b\tilde{*}\phi$  need not be univalent, where the functions $f_a$ and $F_b$ are respectively given in the  Examples \ref{exam42}-\ref{exam42a}. However, the results are true respectively for the function $\phi\in\mathcal{K}$ and  $\mathcal{S}_3$. In fact we have the following the results.

\begin{corollary} \label{cor2a}
For $a\geq6$, let the function $ f_a=h+ \bar{g}$ be the harmonic mapping given in Example \ref{exam42}.  If the function $\phi\in\mathcal{K}$, then the convolution $f_a\tilde{\ast}\phi \in \mathcal{S}_H$ and is convex in the direction of real axis.
\end{corollary}

\begin{proof}
We have $h(z)-g(z)=l(z)=z/(1-z)$. Also, for $a\geq6$, we see that
\begin{align*}
 \RE\left\lbrace\frac{(1-z)^2}{z} \mathcal{D}h(z)\right\rbrace &=\RE\left\lbrace(1-z)^2h'(z)\right\rbrace\\&=\RE\left\lbrace 1+\frac{z}{a}(2-z)\right\rbrace>1/2 \quad\text{for }z\in\mathbb{D}.
\end{align*}
By  Remark \ref{remak39},     the result follows.
\end{proof}

\begin{corollary}\label{corl3a}
Let the function $ F_b=h+ \bar{g}$ be the harmonic mapping given in Example ~\ref{exam42a}. If $\phi\in\mathcal{S}_3$, then $F_b\tilde{\ast}\phi \in \mathcal{S}_H^0$ and is convex in the the direction of real axis.
\end{corollary}

\begin{proof}
  From the definition of $ F_b$, we have \[h(z)-g(z)= \frac{z+z^2}{(1-z)^3},\] and  \[ \mathcal{D}\bigg( \frac{z+bz^2}{(1-z)^2} \bigg) =\frac{z+(1+2b)z^2}{(1-z)^3}= h(z).\] Therefore, for $|b|\leq1/2$,
\begin{align*}
 \RE\left\lbrace\frac{(1-z)^2}{z}\left( h(z)*\log\frac{1}{(1-z)}\right)\right\rbrace &=\RE\left\lbrace\frac{(1-z)^2}{z}\left( \mathcal{D} \frac{z+bz^2}{(1-z)^2} *\log\frac{1}{(1-z)}\right)\right\rbrace\\&=\RE\left\lbrace\frac{(1-z)^2}{z} \frac{z+bz^2}{(1-z)^2}\right\rbrace\\ &=\RE (1+bz) > 1/2.
\end{align*}
The result now follows from Remark \ref{remak40}.
\end{proof}

 For $ 0\leq \alpha <2\pi $, let $\mathcal{S}^0(H_\alpha)\subset\mathcal{S}_H^0 $ denote the class of all harmonic mappings that maps $\mathbb{D} $ onto $H_\alpha $, where\[ H_\alpha :=\left\lbrace z \in \mathbb{C}:\RE (e^{\textit{i}\alpha}z)>-\frac{1}{2}\right\rbrace .\]In \cite{7}, it is shown that if $f=h+\bar{g}\in\mathcal{S}^0(H_\alpha),$ then \begin{equation}\label{eq3}
 h(z)+e^{-2\textit{i}\alpha}g(z)=\frac{z}{1-e^{\textit{i}\alpha} z}.
 \end{equation}
 Using this result, we check the convexity of the convolution $f \tilde{*} \phi$, where the function $f\in\mathcal{S}^0(H_\alpha)$ and the function $\phi\in\mathcal{K}.$ Before this, we give an example of a mapping in class $H_\alpha$.

\begin{example}\label{exam5}In  \cite{3}, it is shown that the harmonic right half-plane mapping $L=M+\overline{N}$, where  \[M(z)=\frac{z-z^2/2}{(1-z)^2} \quad \text {and} \quad N(z)=-\frac{z^2/2}{(1-z)^2},\] maps $\mathbb{D}$ onto the right-half plane $\left\lbrace w:\RE w>-1/2\right\rbrace.$ Consider the mapping $f_{\alpha}=h_{\alpha}+\bar g_{\alpha}$, where \[h_{\alpha}(z)=\frac{z-e^{\textit{i}\alpha}z^2/2}{(1-e^{\textit{i}\alpha}z)^2} \quad \text{and} \quad g_{\alpha}(z)=-\frac{e^{3\textit{i}\alpha}z^2/2}{(1-e^{\textit{i}\alpha}z)^2}.\] Now, for $w=e^{\textit{i}\alpha}z$, we have
\begin{align*}
 e^{\textit{i}\alpha}f_{\alpha}(z)=e^{\textit{i}\alpha}h_{\alpha}(z)+\overline{e^{-\textit{i}\alpha}g_{\alpha}(z)}=&\frac{ze^{\textit{i}\alpha}-e^{2\textit{i}\alpha}z^2/2}{(1-e^{\textit{i}\alpha}z)^2}+\overline{\frac{-e^{2\textit{i}\alpha}z^2/2}{(1-e^{\textit{i}\alpha}z)^2}}\\&=\frac{w-w^2/2}{(1-w)^2}+\overline{\frac{-w^2/2}{(1-w)^2}}=L(w).
\end{align*}
Therefore the function $e^{\textit{i}\alpha}f_{\alpha}$ maps $\mathbb{D}$ onto the right-half plane $\left\lbrace w':\RE w'>-1/2\right\rbrace$. Hence the function  $f_{\alpha}\in \mathcal{S}^0(H_\alpha).$
\end{example}

\begin{example}\label{exam7}
Consider the function $\phi=z+z^2/2\in\mathcal{S}^*$. Then, for the harmonic  mapping $f_{\alpha}$ given in Example \ref{exam5}, the convolution $f_{\alpha}\tilde{*}\phi$ is given by\[f_{\alpha}\tilde{*}\phi(z)=(h_{\alpha}*\phi)(z)+\overline{(g_{\alpha}*\phi)(z)}=z+\frac{3}{4}\emph{e}^{\textit{i}\alpha}z^2-\frac{1}{2}\overline{\emph{e}^{3\textit{i}\alpha}z^2}.\]Its Jacobian $\mathit{J}_{f_{\alpha}\tilde{*}\phi},$  given by\[\mathit{J}_{f_a\tilde{*}\phi}(z)=|(h_{\alpha}*\phi)'(z)|^2-|(g_{\alpha}*\phi)'(z)|^2=1+2|z|^2+3\RE(\emph{e}^{\textit{i}\alpha} z),\]vanishes at $z=-\emph{e}^{-\textit{i}\alpha}/2\in\mathbb{D}.$\

Example \ref{exam7} shows that for the function $\phi\in\mathcal{S}^*$, the convolution $f_{\alpha}\tilde{*}\phi$ is not univalent, where $f_{\alpha}$ is given in Example \ref{exam5}. However, the result is true for the function $\phi\in\mathcal{K}$. In fact we have the following strong result.
\end{example}

 \begin{theorem}\label{theom9a}
 Let the harmonic mapping $f=h+\bar{g} \in\mathcal{S}^0(H_\alpha)$. If the function $\phi\in\mathcal{K}$,
 then the convolution  $f \tilde{*} \phi\in \mathcal{S}_H^0$ and is convex in the direction $ \pi/2-\alpha.$
 \end{theorem}

 \begin{proof} Since the harmonic mapping $f=h+\bar g \in\mathcal{S}^0(H_\alpha)$, from equation \eqref{eq3}, we have \begin{equation}\label{eq4a}
 h(z)+e^{-2\textit{i}\alpha}g(z)=\frac{z}{(1-e^{\textit{i}\alpha} z)}
 \end{equation}
  or\begin{equation}\label{eq4b}
  h(z)-e^{-2\textit{i}(\alpha-\pi/2)}g(z)=\frac{z}{(1-e^{\textit{i}\alpha} z)}.
  \end{equation} Upon differentiating \eqref{eq4a} and writing the dilatation   $g'/h'$ of $f$ by $\omega$, we get $$h'(z)=\frac{1}{(1-e^{\textit{i}\alpha} z)^2(1+e^{-2\textit{i}\alpha} \omega(z))}.$$Therefore
\begin{equation}\label{eq4}
 \RE\frac{(1-e^{\textit{i}\alpha} z)^2}{z} \mathcal{D}h(z)=\RE\frac{1}{1+e^{-2\textit{i}\alpha} \omega(z)}>\frac{1}{2} \quad\text{for }z\in\mathbb{D}.
\end{equation}
 Using  \eqref{eq4b} and \eqref{eq4} in  Remark \ref{remak39}, we see that the convolution $f \tilde{*} \phi\in \mathcal{S}_H^0$ and is convex in the direction  $\pi/2-\alpha$.
  \end{proof}

  In the next result, we show that the convolution of the harmonic mapping $f_{\alpha}=h_{\alpha}+\overline{g_{\alpha}}$, given in Example \ref{exam5}, with the mappings in class $\mathcal{S}_2$ is convex in two perpendicular directions.

  \begin{theorem}\label{theomab}
  Let the function $f_{\alpha}=h_{\alpha}+\overline{g_{\alpha}}$ be the harmonic mapping defined in Example \ref{exam5}. If the function $\phi\in\mathcal{S}_2$, then the convolution $f_{\alpha}*\phi\in\mathcal{S}_H^0$ and is convex in the directions $-\alpha$ and $\pi/2-\alpha$.
\end{theorem}

\begin{proof}
From Example \ref{exam5}, we have
\begin{equation}\label{eqab1}
h_{\alpha}(z)-e^{-2\textit{i}\alpha}g_{\alpha}(z)=\frac{z}{(1-e^{\textit{i}\alpha}z)^2},
\end{equation}
 and
\begin{equation}\label{eqab2}
 \RE\left\lbrace\frac{(1-e^{\textit{i}\alpha}z)^2}{z}h_{\alpha}(z)\right\rbrace=1-\frac{1}{2}\RE (e^{\textit{i}\alpha}z)>\frac{1}{2}\quad \text{for }z \in\mathbb{D}.
\end{equation}
 Also, the function $f_{\alpha}\in\mathcal{S}^0(H_\alpha)$. Therefore, \eqref{eq3} gives
\begin{equation}\label{eqab3}
 h(z)-e^{-2\textit{i}(\alpha-\pi/2)}g(z)=\frac{z}{1-e^{\textit{i}\alpha} z}.
\end{equation}
Since the functions $z/(1-e^{\textit{i}\alpha}z)^2$ and $z/(1-e^{\textit{i}\alpha}z)$ are starlike, in view of \eqref{eqab1}, \eqref{eqab2}, \eqref{eqab3} and Remark \ref{remak38}, the result follows.
\end{proof}

\section{Partial sums of functions in class $\mathcal{S}_n$ and their convolution\\ with harmonic mappings}

For $m$, $n\in\mathbb{N}$, we define the $n{th}$ partial sum of the function $f_m(z)= \sum_{l=1}^{\infty} z^l/l^m $ by $f_{m,n}(z):=\sum_{l=1}^n z^l/l^m$. We now investigate the starlikeness nature of these partial sums.

\begin{lemma}\label{theom5a} For $m$, $n\in\mathbb{N}$, the partial sum $f_{m,n}(z)=\sum_{l=1}^n z^l/l^m$ of the function $f_m(z)= \sum_{l=1}^{\infty} z^l/l^m $ satisfies the following:
\begin{itemize}
\item[(1)]$f_{1,2}(z)= z+\frac{z^2}{2} \in\mathcal{S^*}$;
\item[(2)]$f_{2,3}(z)= z+\frac{z^2}{2^2}+ \frac{z^3}{3^2}\in\mathcal{S^*}$;
\item[(3)]$f_{3,n}(z)=\sum_{l=1}^n z^l/l^3\in\mathcal{S^*}$, for all $n\in \mathbb{N} $.
\end{itemize}
\end{lemma}

\begin{proof}
$(1)$ and $(2)$ follows from Theorem \ref{theomd}. For $l\in\mathbb{N}$, let $a_l=1/l^3$ for $1\leq l\leq n$ and $a_l=0$ for all $l>n$. Then, we have \begin{align*}
\sum_{l=2}^\infty l|a_l|&<-1+\sum_{n=1}^\infty\frac{1}{n^2}=-1+\frac{\pi^2}{6}<1.
\end{align*}Therefore, $(3)$ follows from Theorem \ref{theomd}.
\end{proof}

Using $\mathcal{D}^mf_{{m+l},n}= f_{l,n}$, Lemma \ref{theom5a} and equation \eqref{eq0} gives the following:

\noindent
\begin{lemma}\label{remak1b}The partial sums satisfies:
\begin{itemize}
\item[(1)]$f_{m,2}(z)= z+\frac{z^2}{2^{m}} \in\mathcal{S}_{m-1}$,  $\text{if m}\geq{1}$;
\item[(2)]$f_{m,3}(z)= z+\frac{z^2}{2^{m}}+ \frac{z^3}{3^m}\in\mathcal{S}_{m-2}$, $ \text{if m}\geq{2}$;
\item[(3)]$f_{m,n}(z)= z+\frac{z^2}{2^m}+ \frac{z^3}{3^m}+\dots+\frac{z^n}{n^m} \in\mathcal{S}_{m-3}$, for all $n\in\mathbb{N}, \text{if m}\geq{3}$.
\end{itemize}
\end{lemma}

\begin{lemma}\label{theom7a}
Let $\phi_p$ denotes the $p{th}$ partial sum of the function $ \phi\in \mathcal{S}_m$. Then, we have the following:
\begin{itemize}
\item[(1)] $\phi_2 \in \mathcal{S}_{m-2}$,   if $m\geq{2}$;
\item[(2)] $\phi_3 \in \mathcal{S}_{m-3}$,  if $m\geq{3}$;
\item[(3)] $\phi_p \in \mathcal{S}_{m-4}$,    for all   $p\in\mathbb{N}$, if $m\geq{4}$.
\end{itemize}
\end{lemma}

\begin{proof}
Let $\phi(z)= z+ a_2z^2+a_3z^3+\dots$, $z\in \mathbb{D}$. Then we can write $\phi_p $ as
\begin{align*}
\phi_p(z)&= z+ a_2z^2+a_3z^3+\dots+a_pz^p\\ &=\left(z+2^n a_2z^2+3^na_3z^3+\dots\right)*\bigg(z+\frac{z^2}{2^n}+ \frac{z^3}{3^n}+\dots+\frac{z^p}{p^n}\bigg)\\&= \mathcal{D}^n\phi(z)*f_{n,p}(z).
\end{align*}
Since the function  $ \phi\in \mathcal{S}_m$, \eqref{eq0} shows that the function $\mathcal{D}^n\phi \in \mathcal{S}_{m-n}$. Therefore, by using Lemma \ref{remak1b} and Lemma \ref{theom1a}, we get the result.
\end{proof}

Let the function $f=h +\bar{g}$ be a harmonic mapping, where \[h(z)=\sum_{k=1}^{\infty}a_kz^k\quad \text{and}\quad  g(z)=\sum_{k=1}^{\infty}b_kz^k.\] We define the $n{th}$ partial sum  of $f $ by\[f_n(z):=\sum_{k=1}^{n}a_kz^k+\overline{\sum_{k=1}^{n}b_kz^k}.\]
Therefore, we can write $f_n=f\tilde{\ast}l_n$, where $l_{n}(z)= \sum_{k=1}^{n}z^k$ is the $n{th}$ partial sum of right-half plane mapping  $l(z)=z/(1-z)$.

\begin{theorem}\label{theom8a}
For $n\in\mathbb{N}$, let the function $f=h +\bar{g}$ be a  harmonic mapping in $\mathbb{D}$ with  \begin{equation}\label{eq8b}
 h(z)-g(z)=\left(\frac{z}{(1-z)^2}\right)_*^{n-1}\quad\text{for z} \in\mathbb{D},
\end{equation}
  and
\begin{equation}\label{eq8c}
  \RE{\frac{(1-z)^2}{z}}\left\lbrace h(z)*\bigg(\log\frac{1}{1-z}\bigg)_*^{n-2}\right\rbrace> 1/2\quad\text{ for z}  \in \mathbb{D}.
\end{equation}
  Also, let the function $\phi\in \mathcal{S}_m$. Then, for the partial sum  $(f\tilde{\ast}\phi)_p$ of  the convolution $f\tilde{\ast}\phi$, we have the following:
\begin{itemize}
\item[(1)] If $m\geq n+2 $, then $(f\tilde{\ast}\phi)_2 \in \mathcal{S}_H^0$ and is convex in the the direction of real axis;
\item[(2)] If $m\geq n+3$, then $(f\tilde{\ast}\phi)_2$, $(f\tilde{\ast}\phi)_3 \in \mathcal{S}_H^0$ and are convex in the the direction of real axis;
\item[(3)] If $m\geq n+4$, then $(f\tilde{\ast}\phi)_p \in \mathcal{S}_H^0$ and is convex in the the direction of real axis for all $p\in\mathbb{N}.$
\end{itemize}
\end{theorem}

\begin{proof}
We know that $(f\tilde{\ast}\phi)_p =(f\tilde{\ast}\phi)\tilde{\ast}l_p=f\tilde{\ast}(\phi*l_p)= f\tilde{\ast}\phi_p$, where $\phi_p$ is the  $p{th}$ partial sum of the function $\phi$. Therefore, in order to apply Theorem \ref{theom4a}, we need  $\phi_p$ to be in the class $\mathcal{S}_n$, which follows from Lemma \ref{theom7a}. The result now follows by Theorem \ref{theom4a}.
\end{proof}

\begin{corollary}
 For $a\geq6$, let the   function $ f_a=h+ \bar{g}$  be the harmonic mapping given in Example \ref{exam42}. Then, we have the following:
\begin{itemize}
\item[(1] If the function $\phi \in \mathcal{S}_3$, then  $(f_a\tilde{*}\phi)_2\in\mathcal{S}_H^0$ and is convex in the direction of real axis;
\item[(2] If the function $\phi \in \mathcal{S}_4$, then $(f_a\tilde{*}\phi)_2$, $(f_a\tilde{*}\phi)_3\in\mathcal{S}_H^0$ and are convex in the direction of real axis;
\item[(3] If the function $\phi \in \mathcal{S}_5$, then $(f_a\tilde{*}\phi)_p\in\mathcal{S}_H^0$ and is convex in the direction of real axis for all $p\in\mathbb{N}.$
\end{itemize}
\end{corollary}

\begin{proof}
 We have $h(z)-g(z)=l(z)=z/(1-z)$. Also, for $a\geq6$ we have
\begin{align*}
 \RE\frac{(1-z)^2}{z}\left(h(z)*\left(\frac{z}{(1-z)^2}\right)\right) &=\RE(1-z)^2h'(z)\\&=\RE\left( 1+\frac{z}{a}(2-z)\right)>\frac{1}{2}.
\end{align*}
Therefore, the mapping  $f_a$ satisfies \eqref{eq8b} and \eqref{eq8c} with $n=1$. Hence, the result follows from the Theorem \ref{theom8a}.
\end{proof}

\begin{corollary}\label{corl14a}
 For the harmonic Koebe mapping $K= H+\overline{G}$ and the harmonic half-plane mapping $ L = M+ \overline{N}$, we have the following:
\begin{itemize}
\item[(1] If the function $\phi \in \mathcal{S}_4$, then $(K\tilde{*}\phi)_2$, $(L\tilde{*}\phi)_2\in\mathcal{S}_H^0$ and are convex in the direction of real axis;
\item[(2] If the function $\phi \in \mathcal{S}_5$, then $(K\tilde{*}\phi)_2$, $(K\tilde{*}\phi)_3$, $(L\tilde{*}\phi)_2$, $(L\tilde{*}\phi)_3\in\mathcal{S}_H^0$ and are convex in the direction of real axis;
\item[(3] If the function $\phi \in \mathcal{S}_6$, then $(K\tilde{*}\phi)_p$, $(L\tilde{*}\phi)_p\in\mathcal{S}_H^0$ and are convex in the direction of real axis for all $p\in\mathbb{N}$.
\end{itemize}
\end{corollary}

\begin{proof}
The harmonic Koebe mapping  $K(z)= H(z)+\overline{G(z)}$ is given by \[H(z)=\frac{z-z^2/2+z^3/6}{(1-z)^3} \quad\text{ and }\quad G(z)=\frac{z^2/2+z^3/6}{(1-z)^3}.\] Therefore, \[H(z)-G(z)=\frac{z}{(1-z)^2}\] and

\[\RE\frac{(1-z)^2}{z}H(z) =\RE\frac{1-z/2+z^2/6}{1-z} =\RE\bigg(\frac{2/3}{1-z} + \frac{1}{3} + \frac{z}{6}\bigg) >\frac{1}{2}.\]
Also, the half plane mapping $L(z)= M(z)+\overline{N(z)}$ is given by \[M(z)=\frac{z-z^2/2}{(1-z)^2}\quad \text{and} \quad N(z)=\frac{-z^2/2}{(1-z)^3}.\] Then, \[M(z)-N(z)=\frac{z}{(1-z)^2}\quad\text{and}\quad \RE\frac{(1-z)^2}{z}M(z) =\RE\left(1-\frac{z}{2}\right)>\frac{1}{2}.\]
Therefore, the functions, $K$ and $L$ satisfies \eqref{eq8b} and \eqref{eq8c} with $n=2$. Hence, the result follows from the Theorem \ref{theom8a}.
\end{proof}

\begin{example}
For $m\in\mathbb{N}$, we can easily see that the function \[f_m(z):= z+\frac{z^2}{2^m}+ \frac{z^3}{3^m}+\dots \in \mathcal{S}_{m+1}.\]
Therefore, by Corollary \ref{corl14a}, the following functions
\begin{itemize}
\item[(1)]$(K\tilde{*}f_3)_2(z) =z+\frac{5}{16}z^2 +\frac{1}{16}\bar{z}^2$;
\item[(2)]$(K\tilde{*}f_4)_3(z) =z+\frac{5}{32}z^2+\frac{14}{243}z^3 +\frac{1}{32}\bar{z}^2+\frac{5}{243}\bar{z}^3 $;
\item[(3)]$(K\tilde{*}f_5)_4(z) =z+\frac{5}{64}z^2+\frac{14}{729}z^3 +\frac{15}{2048}z^4+\frac{1}{64}\bar{z}^2+\frac{5}{729}\bar{z}^3+ \frac{7}{2048}\bar{z}^4$
\end{itemize}
 belongs to $\mathcal{S}_H^0$ and are convex in the direction of real axis.
\end{example}

\begin{corollary}
For $|b|\leq1/2$, let the function $ F_b=h+ \bar{g}$ be the harmonic mapping given in Example \ref{exam42a}. Then, we have the following:
\begin{itemize}
\item[(1] If the function $\phi \in \mathcal{S}_5$, then $(F_b\tilde{*}\phi)_2\in\mathcal{S}_H^0$ and is convex in the direction of real axis.
\item[(2] If the function $\phi \in \mathcal{S}_6$, then $(F_b\tilde{*}\phi)_2$, $(F_b\tilde{*}\phi)_3\in\mathcal{S}_H^0$ and are convex in the direction of real axis.
\item[(3] If the function $\phi \in \mathcal{S}_7$, then $(F_b\tilde{*}\phi)_p\in\mathcal{S}_H^0$ and are convex in the direction of real axis for all $p\in\mathbb{N}$.
\end{itemize}
\end{corollary}

\begin{proof}
 We have \[h(z)-g(z)= \frac{z+z^2}{(1-z)^3} = \left(\frac{z}{(1-z)^2}\right)_*^{2},\] and  \[ \mathcal{D} \left(\frac{z+bz^2}{(1-z)^2}\right)  =\frac{z+(1+2b)z^2}{(1-z)^3}= h(z).\] Using above equation, we get, for $|b|\leq1/2$,
\begin{align*}
 \RE\frac{(1-z)^2}{z}\left(h(z)*\log\frac{1}{(1-z)}\right) &=\RE\frac{(1-z)^2}{z}\left( \mathcal{D} \frac{z+bz^2}{(1-z)^2} *\log\frac{1}{(1-z)}\right)\\&=\RE\frac{(1-z)^2}{z} \frac{z+bz^2}{(1-z)^2}\\ &=\RE {(1+bz)} > 1/2.
\end{align*}
Therefore, the function $F_b$ satisfies \eqref{eq8b} and \eqref{eq8c} with $n=3$. Hence, the result follows from the Theorem \ref{theom8a}.
\end{proof}

\begin{theorem}
For  the mapping $f=h+\bar{g} \in\mathcal{S}^0(H_\alpha)$,
 we have the following:
\begin{itemize}
\item[(1] If the function $\phi \in \mathcal{S}_3$, then $(f\tilde{*}\phi)_2\in \mathcal{S}_H^0$ and is convex in the direction $(\pi/2-\alpha).$
\item[(2] If the function $\phi \in \mathcal{S}_4$, then $(f\tilde{*}\phi)_2$, $(f\tilde{*}\phi)_3\in \mathcal{S}_H^0$ and are convex in the direction $(\pi/2-\alpha).$
\item[(3] If the function $\phi \in \mathcal{S}_5$, then $(f\tilde{*}\phi)_p\in \mathcal{S}_H^0$ and are convex in the direction $(\pi/2-\alpha)$ for all $p\in\mathbb{N}.$
\end{itemize}
\end{theorem}

 \begin{proof}
Since $(f\tilde{\ast}\phi)_p =f\tilde{\ast}\phi_p$, where $\phi_p$ is the  $p{th}$ partial sum of the function $\phi$, therefore, in order to apply Theorem \ref{theom9a},  we need the function $\phi_p$ to be in the class $\mathcal{K}$, which follows from  Lemma \ref{theom7a}. The result now follows from Theorem \ref{theom9a}.
 \end{proof}

\begin{theorem}
 For the harmonic mapping  $f_{\alpha}$  defined in the Example \ref{exam5}, we have the following:
\begin{itemize}
\item[(1] If the function $\phi \in \mathcal{S}_4$, then $(f_{\alpha}\tilde{*}\phi)_2\in \mathcal{S}_H^0$ and is convex in the directions $-\alpha$ and $(\pi/2-\alpha).$
\item[(2] If the function  $\phi \in \mathcal{S}_5$, then $(f_{\alpha}\tilde{*}\phi)_2$, $(f_{\alpha}\tilde{*}\phi)_3\in \mathcal{S}_H^0$  and are convex in the directions $-\alpha$ and $(\pi/2-\alpha).$
\item[(3] If the function $\phi \in \mathcal{S}_6$, then $(f_{\alpha}\tilde{*}\phi)_p\in \mathcal{S}_H^0$ and is convex in the directions $-\alpha$ and $(\pi/2-\alpha)$ for all $p\in\mathbb{N}$.
\end{itemize}
\end{theorem}

 \begin{proof}
Since $(f_{\alpha}\tilde{\ast}\phi)_p= f_{\alpha}\tilde{\ast}\phi_p$, where $\phi_p$ is the  $p{th}$ partial sum of the function $\phi$, therefore, in order to apply Theorem \ref{theomab},  we need the function $\phi_p$ to be in the class $\mathcal{S}_2$, which follows from Lemma \ref{theom7a}. The results now follows from Theorem \ref{theomab}.
 \end{proof}

\end{document}